\documentclass[11pt,oneside]{amsart}
\usepackage{amssymb, amsmath, amsthm}
\usepackage{color}
\usepackage{fullpage}
\usepackage{xifthen}
\usepackage{graphicx}
\usepackage{tikz}
\usepackage{tabularx}

\newtheorem{theorem}{Theorem}
\newtheorem{lemma}[theorem]{Lemma}

\newtheorem{proposition}[theorem]{Proposition}

\theoremstyle{definition}
\newtheorem{definition}[theorem]{Definition}
\theoremstyle{remark}
\newtheorem{remark}[theorem]{Remark}

\renewcommand{\b}[1]{\textbf{#1}}	
\newcommand{\lat}{\mathcal{L}}		
\newcommand{\E}{\mathbb{E}}		
\renewcommand{\P}{\mathbb{P}}		
\newcommand{\real}{\mathbb{R}}		
\newcommand{\var}{\mathrm{Var}}	

\begin{document}

\title{Intrinsic Volumes of Random Cubical Complexes}

\author{Michael Werman}
\address{The Hebrew University of Jerusalem, Israel}
\email{werman@cs.huji.ac.il}

\author{Matthew L. Wright}
\address{Institute for Mathematics and its Applications, University of Minnesota, USA}
\email{mlwright@ima.umn.edu}

\date{\today}

\begin{abstract}
Intrinsic volumes, which generalize both Euler characteristic and Lebesgue volume, are important properties of $d$-dimensional sets. 
A random cubical complex is a union of unit cubes, each with vertices on a regular cubic lattice, constructed according to some probability model.
We analyze and give exact polynomial formulae, dependent on a probability, for the expected value and variance of the intrinsic volumes of several models of random cubical complexes. 
We then prove a central limit theorem for these intrinsic volumes.
For our primary model, we also prove an interleaving theorem for the zeros of the expected-value polynomials.
The intrinsic volumes of cubical complexes are useful for understanding the shape of random $d$-dimensional sets and for characterizing noise in applications.
\end{abstract}



\maketitle

\section{Introduction}
In the literature about random combinatorial objects, a number of recent papers analyze topological properties of various random structures.
In particular, Kahle studied the Betti numbers of simplicial complexes obtained from random points in $\real^d$ \cite{Kahle11, Kahle13}, and 
Kahle and Meckes gave limit theorems for such Betti numbers \cite{KaMe}.
Linial and Meshulam studied homology groups of random $2$-dimensional complexes \cite{LiMe}, while Meshulam and Wallach generalized to higher-dimensional complexes \cite{MeWa}.
Cohen et.\ al.\ obtained results about collapsibility and topological embeddings of the Linial-Meshulam complexes \cite{CCFK}.

Among these papers on random complexes, topological properties (e.g.\ Euler characteristic, homology, Betti numbers) are prevalent, but the intrinsic volumes are scarcely to be found.
The intrinsic volumes\footnote{Intrinsic volumes known by various names, including \emph{Hadwiger measures}, \emph{Minkowski functionals}, and \emph{quermassintegrale}. These concepts are equivalent up to normalization.} generalize both Euler characteristic and Lebesgue volume, providing information about not only the topology but also the geometry of sets.
In contrast to the Betti numbers, the intrinsic volumes are additive (local), which makes them easier to compute in some cases.
The intrinsic volumes are useful in applications such as imaging and stereology, and can be computed quickly for low-dimensional binary images in pixel (or voxel) form \cite{gray,GKMSS,LMB,SON,Svane}. 
Furthermore, the intrinsic volumes arise in Adler and Taylor's work on random fields \cite{AdTa} (see also \cite{Wright}), and various applications of the Euler characteristic to image classification can be found in \cite{RW}.

This paper analyzes the intrinsic volumes of several models of random cubical complexes.
A random cubical complex is a union of unit cubes with vertices on a regular $d$-dimensional cubic lattice.
These complexes are relevant for digital images and discretized domains in which data is represented on a cubical lattice of possibly high dimension.
This work is motivated in part by the desire to understand noise in digital images.
The construction of a random cubical complex depends on the particular model and a probability.
We give polynomial formulae, dependent on probability $p$, for the expected values and variances of the intrinsic volumes of these complexes.
We analyze these polynomials and, for our primary model, prove an interleaving theorem about their zeros.
Moreover, we prove that as the size of the complexes tend toward infinity, the distribution of a particular intrinsic volume approaches a normal distribution.

We first provide background about the intrinsic volumes in Section \ref{sec:int_vol}.
Section \ref{sec:cubes} describes rigorously our models of random cubical complexes.
We then analyze our primary model, the \emph{voxel model}, giving expected values and variances in Sections \ref{sec:voxel_exp} and \ref{sec:voxel_var}, respectively.
Section \ref{sec:CLT} contains a central limit theorem for the intrinsic volumes.
In Section \ref{sec:models}, we turn to other models of random cubical complexes.

\section{Background: Intrinsic Volumes}\label{sec:int_vol}

The intrinsic volumes are \emph{valuations}.\footnote{Valuations are not \emph{measures}, though the concepts are similar. In particular, the intrinsic volumes may take on negative values and are not monotonic.}
A \b{valuation} is a real-valued function on a class of sets that provides a notion of \emph{size} for each set and satisfies the additive property.
The \b{additive property} says that if $v$ is a valuation, then $v(X) + v(Y) = v(X \cap Y) + v(X \cup Y)$ for sets on which $v$ is defined.
The \b{intrinsic volumes} are valuations, related to Euler characteristic and Lebesgue volume, normalized to be \emph{intrinsic} to sets and independent of the ambient dimension in which sets may be embedded.
We give a brief introduction to the intrinsic volumes here; the interested reader may pursue the references for more details.

Defined on various classes of sets in Euclidean space $\real^d$, the intrinsic volumes are denoted $\mu_0, \mu_1, \ldots, \mu_d$.
Intuitively, $\mu_k$ assigns a real number to each set that provides a notion of the $k$-dimensional size of the set \cite{Schanuel}.
The $0$-dimensional valuation $\mu_0$ is the Euler characteristic, the only topological invariant among the intrinsic volumes.
The intrinsic volume $\mu_1$ gives a notion of the \emph{length} of a set; $\mu_2$ gives a notion of \emph{area}, etc.
For example, if $X$ is a compact, convex $3$-dimensional set, then $\mu_2(X)$ is one-half the surface area of $X$.
For a $d$-dimensional set, $\mu_d$ gives $d$-dimensional Lebesgue volume. 
The intrinsic volumes are invariant with respect to rigid motions (translations and rotations) of sets.
Furthermore, the classic Hadwiger Theorem says that the intrinsic volumes form a basis for the vector space of all rigid-motion invariant valuations that are continuous on convex sets with respect to the Hausdorff metric.

The intrinsic volumes are commonly defined on the class of compact convex sets in $\real^d$, but it is possible to be much more general.
In their  text, Klain and Rota approach  intrinsic volumes via valuations on a lattice of subsets of $\real^d$ \cite{KlRo}.
Another common approach defines the intrinsic volumes on compact convex sets via the Steiner tube formula \cite{ScWe}.
Hadwiger's formula can be used to define intrinsic volumes for all sets in an o-minimal structure, including sets that are neither closed nor convex \cite{BGW, vdD}. 
We do not need the full generality of these approaches in this paper, since it suffices for our present purposes to define the intrinsic volumes for finite unions of cubes.
The intrinsic volumes of a closed unit cube, given in Definition \ref{def:int_vol_cl}, agree with the intrinsic volumes as defined in any of the above sources.

\begin{definition}\label{def:int_vol_cl}
	A \b{closed $i$-cube} shall refer to any translate of the unit cube $[0,1]^i$, for any nonnegative integer $i$.
	If $X$ is a closed $i$-cube, then the \b{intrinsic volume} $\mu_k$ of $X$ is $\mu_k(X) = \binom{i}{k}$.
	Note that if $k>i$, then $\mu_k(X)=0$.
\end{definition}

When the intersection $X \cap Y$ of two closed cubes $X$ and $Y$ (of possibly different dimension) is also a closed cube, then the intrinsic volume of the union $\mu_k(X \cup Y)$ follows from Definition \ref{def:int_vol_cl} and the additive property.
Thus, intrinsic volumes are well-defined for a class of (non-convex) sets obtained from finite unions of closed cubes.

Additivity also induces intrinsic volumes on the interiors of cubes.
An \b{open $i$-cube} is any translate of $(0,1)^i$, for any positive integer $i$.
We may refer to a point as \emph{either} an open or a closed $0$-cube.
The following proposition gives intrinsic volumes of open cubes, which we introduce as an auxiliary notion that will simplify many later calculations.
Furthermore, this proposition gives a special case of the intrinsic volumes on definable sets found in \cite{BGW}.


\begin{proposition}\label{prop:int_vol_open}
	If $X = (0,1)^i$ is an open $i$-cube, then the intrinsic volume $\mu_k$ of $X$ is 
	\begin{equation}
		\mu_k(X) = (-1)^{i-k}\binom{i}{k}.
	\end{equation}
\end{proposition}
\begin{proof}
	Let $X$ be an open $i$-cube.
	Then $X$ has $\binom{i}{j}2^{i-j}$ faces of dimension $j$, for $j = 0, 1, \ldots, i$.
	Inclusion-exclusion (i.e.\ the additive property) lets us write $\mu_k(X)$ as an alternating sum over all closed faces of $X$.
	Since $\mu_k$ is zero for any closed face of dimension less than $k$, we have:
	\begin{equation*}
		\mu_k(X) = \sum_{j=k}^i (-1)^{i-j} \binom{i}{j} 2^{i-j} \mu_k\left( [0,1]^j \right) = \sum_{j=k}^i (-2)^{i-j} \binom{i}{j} \binom{j}{k}.
	\end{equation*}
	We then use the binomial coefficient identity $\binom{i}{j} \binom{j}{k} = \binom{i}{k} \binom{i-k}{i-j}$ to obtain
	\begin{equation*}
		\mu_k(X) = \binom{i}{k} \sum_{j=k}^i \binom{i-k}{i-j} (-2)^{i-j} =  \binom{i}{k} \sum_{t=0}^{i-k} \binom{i-k}{t} (-2)^{t} =\binom{i}{k} (-1)^{i-k}
	\end{equation*}
	as desired.
\end{proof}

The intrinsic volumes of open cubes will appear frequently in sequel, so we assign them special notation.
Let $\mu_{k,i}$ refer to the intrinsic volume $\mu_k$ of an open unit $i$-cube.
That is,
\begin{equation}\label{eq:openvol}
	\mu_{k,i} = (-1)^{i-k}\binom{i}{k}.
\end{equation}

Proposition \ref{prop:int_vol_open} implies that the Euler characteristic of an open $i$-cube $X$ is $\mu_0(X) = (-1)^i$.
We note that $\mu_0$ is sometimes called the \emph{combinatorial} Euler characteristic\footnote{The combinatorial Euler characteristic is not a homotopy invariant, but on compact sets it agrees with the \emph{topological} Euler characteristic (which has value $1$ on any contractible set).} and is found in the literature \cite{BGW, CGR, vdD}.

The cubical complexes that we study in this paper may be decomposed into disjoint unions of open cubes.
This perspective simplifies calculations because additivity of the intrinsic volumes is especially easy to work with over \emph{disjoint} unions.
Throughout the paper, we will consider a cubical complex to be a disjoint union of open cubes, even if the complex itself is topologically closed.
For a closed cubical complex, we could obtain the same results by considering a non-disjoint union of closed cubes, but this would require more bookkeeping via the inclusion-exclusion principle.
Our approach shifts this bookkeeping into the definition of intrinsic volumes of open cubes, while also permitting us to work with cubical complexes that are not necessarily closed (which we do in Section \ref{sec:models}).

\section{Random Cubical Complexes}\label{sec:cubes}

We construct random geometric objects that are subsets of a regular array of $d$-dimensional unit cubes.
Let $n \in \{ 2, 3, 4, \ldots, \}$, and
consider a regular lattice of $d$-dimensional unit cubes, with $n$ cubes in each direction, for a total of $n^d$ unit $d$-cubes.
Let $\lat$ be this lattice of unit cubes.
For simplicity, so as not to bother with the border of $\lat$, we suppose that opposite $(d-1)$-dimensional faces of $\lat$ are identified, so that $\lat$ has the topology of a $d$-dimensional torus.
The number of $i$-dimensional faces of any $d$-cube is $\binom{d}{i} 2^{d-i}$.
Furthermore, any $i$-cube is a face of $2^{d-i}$ $d$-cubes.
Thus, the number of $i$-cubes in $\lat$, which we denote $N_i$, is the number of $d$-cubes, times the number of $i$-cubes per $d$-cube, divided by the number of $d$-cubes of which each $i$-cube is a face:
\begin{equation}\label{eq:i-cubes}
	N_i = \frac{n^d \binom{d}{i} 2^{d-i} }{2^{d-i}} = \binom{d}{i}n^d.
\end{equation}

We randomly select a subset $C$ of $\lat$, such that $C$ is a union of open unit cubes of various dimensions.
Such a subset $C$ we shall call a \b{random cubical complex}.
In general, there is no requirement that $C$ itself be open or closed.
We are interested in understanding the intrinsic volumes $\mu_k(C)$, for all $k=0, 1, \ldots, d$.
Of course, the intrinsic volumes $\mu_k(C)$ depend on how we select the unit cubes that comprise $C$.
Each selection method yields interesting connections among the $\mu_k(C)$.
We now describe the three models of cubical complexes that we will examine.

Our primary selection method we call the \b{voxel model}.
In this model, we select the closed  unit $d$-cubes of $\lat$ independently with equal probability.
Specifically, the unit $d$-cubes of $\lat$ are included in $C$ independently with probability $p$, and a unit $i$-cube of $\lat$, for $i<d$, is included in $C$ exactly when it is a face of an included $d$-cube.
In this case, $C$ is a closed set, consisting of a union of closed $d$-cubes that possibly overlap on some of their faces.
Figure \ref{fig:voxel} illustrates the voxel model in two dimensions for several values of $p$.
We find that the expected value and variance of $\mu_k(C)$ are given by polynomials in $q=1-p$ of degree $2^{d-k}$.
We show various identities between these polynomials and prove an interleaving property of the roots of the expected-value polynomials.

\setlength\fboxsep{0pt}
\setlength\fboxrule{0.5pt}
\begin{figure}[tb]
	\begin{tabular}{ccccc}
		\fcolorbox{gray}{black}{\includegraphics[width=0.15\linewidth]{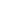}} &
		\fcolorbox{gray}{black}{\includegraphics[width=0.15\linewidth]{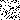}} &
		\fcolorbox{gray}{black}{\includegraphics[width=0.15\linewidth]{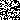}} &
		\fcolorbox{gray}{black}{\includegraphics[width=0.15\linewidth]{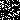}} &
		\fcolorbox{gray}{black}{\includegraphics[width=0.15\linewidth]{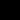}} \\
		$p=0$ & $p=0.25$ & $p=0.5$ & $p=0.75$ & $p=1$
	\end{tabular}
	\caption{Example of the two-dimensional \emph{voxel model} of random cubical complexes ($d=2$, $n=20$). The black set is the cubical complex, which is regarded as a subset of the torus.}
	\label{fig:voxel}
\end{figure}

The voxel model is the most natural for various applications, but we also consider other possible models.
A related model is the \b{plaquette model} which includes all cubes of dimension less than $d$, while including the $d$-cubes independently with probability $p$.
This model is related to the plaquette model of percolation \cite{Aizenman, GrHo}, and is analogous to the models of random simplicial complexes studied by Linial and Meshulam \cite{LiMe}.

Additionally, the following two models allow for randomness to occur in all dimensions.
The \b{closed faces model} allows lower-dimensional closed cubes to be included.
Specifically, this model selects each open $i$-cube independently with probability $p$, and then forms a cubical complex consisting of the closure of all selected cubes.
The \b{independent faces model} includes each open $i$-cube independently with probability depending on $i$, and the resulting cubical complex might not be closed.
For an interesting model, we choose to include open $i$ cubes with probability $p^i$.
These models may be useful for applications in which features of different dimensions occur simultaneously.

We emphasize that regarding our random cubical complexes as subsets of the $d$-torus allows us to ignore boundary effects.
Furthermore, as $n \to \infty$, the distribution of intrinsic volumes of random cubical complexes living in $\real^d$ converges to that of complexes (of the same model) living in the $d$-torus.
That is, boundary effects become negligible as $n \to \infty$.
Thus, since we are primarily interested in behavior as $n \to \infty$, the torus setting is natural and simplifies computation.

\begin{figure}[tb]
	\begin{tikzpicture}
		\matrix [column sep=60pt, row sep=0in, cells={scale=0.4}] {
		
		\foreach \p in {(0,0),(0,1),(0,4),(0,6),(0,7),(1,7),(3,4),(3,5),(3,6),(3,7),(4,4),(4,6),(5,4),(5,5),(5,7),(7,1),(7,3),(7,5),(7,7)}
		{	\draw[thick,fill=gray] \p rectangle +(1,1); }
		
		\foreach \p in {(1,0),(1,1),(1,5),(2,1),(3,0),(5,0),(6,0),(6,4),(6,8),(7,0)}
		{	\draw[thick] \p -- +(1,0); }
		
		\foreach \p in {(0,3),(0,5),(1,2),(2,1),(2,3),(3,2),(3,3),(4,2),(6,1),(7,4),(8,0),(8,4),(8,6)}
		{	\draw[thick] \p -- +(0,1); }
		
		\foreach \p in {(2,0),(2,1),(2,2),(2,3),(2,4),(2,5),(2,7),(2,8),(5,0),(5,1),(5,4),(5,5),(5,6),(5,7),(5,8)}
		{	\node[style={shape=circle,fill=black,scale=0.25}] at \p {}; }
		\foreach \x in {0,1,3,4,6,7,8}
		\foreach \y in {0,...,8}
		{	\node[style={shape=circle,fill=black,scale=0.25}] at (\x,\y) {}; }
		
		&
		
		\foreach \p in {(0,3),(0,4),(0,7),(1,3),(1,4),(2,5),(3,4),(3,7),(4,0),(4,7),(5,4),(6,2),(6,4),(7,0),(7,1),(7,3)}
		{	\fill[gray] \p rectangle +(1,1); }
		
		\draw[dotted, very thick] (0,7) rectangle +(1,1);
		\foreach \p in {(0,5),(1,5),(3,8),(4,0),(4,1),(4,8),(6,2),(6,3),(7,1),(7,4)} 
		{	\draw[dotted, very thick] \p -- +(1,0); }
		\foreach \p in {(1,4),(2,5),(3,7),(4,0),(4,7),(5,4),(5,7),(6,4),(7,0),(7,3),(7,4),(8,0),(8,1)} 
		{	\draw[dotted, very thick] \p -- +(0,1); }
		
		\foreach \p in {(0,1),(0,2),(0,3),(0,4),(0,6),(1,1),(1,3),(1,4),(1,6),(2,2),(2,3),(2,5),(2,6),(3,1),(3,3),(3,4),(3,5),(3,6),(3,7),(4,4),(4,5),(4,6),(4,7),(5,4),(5,5),(5,6),(6,0),(6,4),(6,5),(6,7),(7,0),(7,2),(7,3),(7,6),(6,8),(7,8)}
		{	\draw[thick] \p -- +(1,0); }
		
		\foreach \p in {(0,3),(0,4),(0,5),(1,0),(1,2),(1,3),(2,3),(2,4),(2,6),(3,1),(3,4),(3,5),(4,1),(4,3),(4,4),(4,5),(5,0),(5,3),(5,5),(5,6),(6,0),(6,1),(6,2),(6,3),(6,5),(6,6),(7,1),(7,2),(7,5),(7,7),(8,3),(8,4),(8,5)}
		{	\draw[thick] \p -- +(0,1); }
		
		\foreach \x in {0,...,8}
		\foreach \y in {0,...,8}
		{	\node[style={shape=circle,fill=black,scale=0.25}] at (\x,\y) {}; }
		
		&
		
		\draw[thick] (0,0) grid (8,8);
		\foreach \i in {4,6,9,12,15,27,38,41,42,51,56,57,59,62}
		{
			\pgfmathparse{int(mod(\i,8))}
			\edef\x{\pgfmathresult}
			\pgfmathparse{int(\i/8)}
			\edef\y{\pgfmathresult}
			\draw[fill] (\x,\y) rectangle (\x+1,\y+1);
		}
		
		\\
		};
	\end{tikzpicture}
	\caption{The \emph{closed faces model} (left) selects each open $i$-cube independently with probability $p$, and then forms a cubical complex consisting of the closure of all selected cubes. In the illustration, $d=2$, $n=8$, and $p=0.25$. 
		The \emph{independent faces model} (center) includes each open $i$-cube independently with probability $p^i$. The diagram shows $d=2$, $n=8$, and $p=0.5$. 
		The \emph{plaquette model} (right) includes top-dimensional cubes with probability $p$ and includes all cubes of lower dimension; here $d=2$, $n=8$, and $p=0.25$. }
	\label{fig:models}
\end{figure}
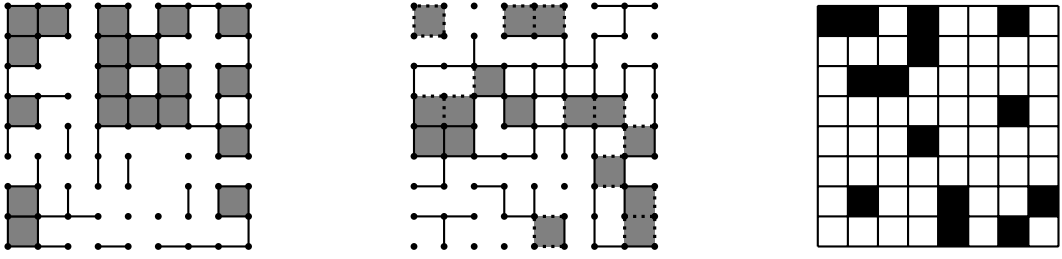

\section{Voxel Model: Expected Value}\label{sec:voxel_exp}

We now compute the expected value of $\mu_k(C)$ for the voxel model.
Recall that this model selects $d$-cubes from $\lat$ independently with probability $p$, and that $C$ is the union of the closure of all selected cubes.
Since $\lat$ is a disjoint union of open cubes, it suffices to sum the expected value of $\mu_k(X)$ over all open cubes $X$ (of all dimensions) in $\lat$.
For any open $i$-cube $X$, let $\xi_{X,k}$ be the random variable that indicates the contribution of $X$ to $\mu_k(C)$.
That is, $\xi_{X,k} = \mu_{k,i} = (-1)^{i-k}\binom{i}{k}$ if $X$ is included in $C$, and $\xi_{X,k} =0$ otherwise.
When the subscript $k$ is clear, we will write $\xi_X$ instead of $\xi_{X,k}$.
Writing the expected value as a sum over all open cubes, indexed by dimension, we obtain
\begin{equation}\label{eq:expect1}
	\E(\mu_k(C)) = \sum_{\text{open } X \in \lat} \E(\xi_X) = \sum_{i=k}^d N_i P_i \mu_{k,i},
\end{equation}
where $P_i$ is the probability that any particular open $i$-cube is included in $C$.

Since any $i$-cube is a face of $2^{d-i}$ $d$-cubes, the probability that any particular $i$-cube is \emph{not} included in $C$ is $(1-p)^{2^{d-i}}$.
The probability that any particular $i$-cube \emph{is} included in $C$ is then $1- (1-p)^{2^{d-i}}$.
Let $q=1-p$, and we have
\begin{equation*}
	P_i = 1 - q^{2^{d-i}}.
\end{equation*}
Recalling the count $N_i$ of $i$-cubes in $\lat$ from equation \eqref{eq:i-cubes} and the value of $\mu_{k,i}$ from equation \eqref{eq:openvol}, we obtain
\begin{equation}\label{eq:expect2}
	\E(\mu_k(C))= \sum_{i=k}^{d} \binom{d}{i} {n}^{d} \left( 1 - q^{2^{d-i}} \right) (-1)^{i-k} \binom{i}{k}.
\end{equation}
In equation \eqref{eq:expect2}, $n$ only appears as a factor $n^d$, the ($d$-dimensional) volume of $\lat$.
We normalize by volume to remove this factor.
The expected value of $\mu_k(C)$ per unit volume is a polynomial in $q$, which we  denote $E_{d,k}(q)$.
That is,
\begin{equation}\label{eq:expect_normalized}
	E_{d,k}(q) = \frac{1}{n^d}\E(\mu_k(C)) = \sum_{i=k}^{d} (-1)^{i-k} \binom{d}{i} \binom{i}{k} \left( 1 - q^{2^{d-i}} \right).
\end{equation}

The following theorem expresses our computation in a simpler form.

\begin{theorem}
	The expected value of $\mu_k(C)$, normalized by volume, is a polynomial in $q$ of degree $2^{d-k}$, given by the following formula
	\begin{equation}\label{eq:expect_dk}
		E_{d,k}(q) = \frac{1}{n^d}\E(\mu_k(C)) =
		\begin{cases} \sum_{i=k}^{d} (-1)^{i-k+1} \binom{d}{i} \binom{i}{k} q^{2^{d-i}} & \text{if } k < d \\  1-q =p& \text{if } k=d.\end{cases}
	\end{equation}
\end{theorem}
\begin{proof}
	Use the following identity to simplify equation \eqref{eq:expect_normalized}:
	\begin{equation}\label{eq:binom_identity}
		\sum_{i=k}^d (-1)^{i-k} \binom{d}{i} \binom{i}{k} = \begin{cases} 0 & \text{if } k \in \{ 0, 1, \ldots, d-1\} \\ 1 & \text{if } k = d. \end{cases}
	\end{equation}
\end{proof}

The expected value polynomials $E_{d,k}(q)$ appear in Table \ref{tab:exp} for small $d$ and $k$.
We seem to have a family of polynomials indexed by two parameters, $d$ and $k$.
However, in the next section we show that, up to constant multiples, we have a family of polynomials indexed by the single parameter $d-k$, and the roots of these polynomials interleave.

\begin{table}[htb]
	{\small \begin{tabular}{c|c c c c c}
	& $d=0$ & $d=1$ & $d=2$ & $d=3$ & $d=4$\\
	\hline
	$k=0$ & $1-q$ & $q-q^2$ & $-q+2q^2-q^4$ & $q-3q^2+3q^4-q^8$ & $-q+4q^2-6q^4+4q^8-q^{16}$ \\
	\hline
	$k=1$ & & $1-q$ & $2(q-q^2)$ & $3(-q+2q^2-q^4)$ & $ 4(q-3q^2+3q^4-q^8)$ \\
	\hline
	$k=2$ & & & $1-q$ & $3(q-q^2)$ & $6(-q+2q^2-q^4)$ \\
	\hline
	$k=3$ & & & & $1-q$ & $4(q-q^2)$ \\
	\hline
	$k=4$ & & & & & $1-q$ \\
	\end{tabular}}
	\caption{Expected value polynomials $E_{d,k}(q)$, written in factored form to illustrate Proposition \ref{prop:poly}.}
	\label{tab:exp}
\end{table}

\subsection{Expected Value Polynomials}

The polynomial $E_{d,k}(q)$ is a constant multiple of the polynomial $E_{d-k,0}(q)$.
Thus, if we understand the sequence of polynomials $E_{d,0}$ that give the expected Euler characteristics for $d=1, 2, \ldots$, then we understand the expected intrinsic volumes $E_{d,k}$ for any $d$ and $k$.
We can regard the quantity $d-k$ as the \emph{codimension} of the valuation $\mu_k$ on a $d$-dimensional complex, and the roots of the polynomial $E_{d,k}(q)$ depend only on this codimension.
Specifically, we have the following proposition.

\begin{proposition}\label{prop:poly}
	The expected value polynomials satisfy
	\begin{equation}\label{eq:d-k}
		E_{d,k}(q) = \binom{d}{k}E_{d-k,0}(q).
	\end{equation}
\end{proposition}
\begin{proof}
	Use the binomial coefficient identity $\binom{d}{i}\binom{i}{k} = \binom{d}{k}\binom{d-k}{d-i}$ to rewrite equation \eqref{eq:expect_normalized} as
	\begin{equation*}
		E_{d,k}(q) = \binom{d}{k} \sum_{i=k}^{d} (-1)^{i-k} \binom{d-k}{d-i} \left( 1 - q^{2^{d-i}} \right).
	\end{equation*}
	Now let $j=i-k$ to re-index the sum, obtaining
	\begin{equation*}
		E_{d,k}(q) = \binom{d}{k} \sum_{j=0}^{d-k} (-1)^{j} \binom{d-k}{j} \left( 1 - q^{2^{d-k-j}} \right) = \binom{d}{k} E_{d-k,0}(q)
	\end{equation*}
	as desired.
\end{proof}

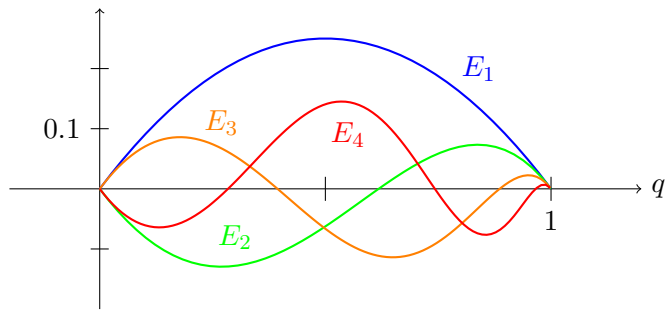
\begin{figure}[bt]
	\begin{center}
		\begin{tikzpicture}[xscale=6,yscale=8]
			\draw[->] (-0.2,0) -- (1.2,0) node[right] {$q$};
			\draw (1,0.02) -- (1,-0.02) node[below] {$1$};
			\draw (0.5,0.02) -- (0.5,-0.02);
			
			\draw[->] (0,-0.2) -- (0,0.3);
			\draw (-0.02,0.1) node[left] {$0.1$} -- (0.02,0.1);
			\draw (-0.02,0.2) -- (0.02,0.2);
			\draw (-0.02,-0.1) -- (0.02,-0.1);
			
			\draw[blue, thick, domain=0:1,samples=40] plot (\x, {\x-\x*\x});
			\node[blue] at (0.84,0.2) {$E_1$};
			\draw[green, thick, domain=0:1,samples=50] plot (\x, {-\x+2*\x*\x-pow(\x,4)});
			\node[green] at (0.3,-0.08) {$E_2$};
			\draw[orange, thick, domain=0:1,samples=100] plot (\x, {\x-3*\x*\x+3*pow(\x,4)-pow(\x,8)});
			\node[orange] at (0.27,0.11) {$E_3$};
			\draw[red, thick, domain=0:1,samples=100] plot (\x, {-\x+4*\x*\x-6*pow(\x,4)+4*pow(\x,8)-pow(\x,16)});
			\node[red] at (0.55,0.09) {$E_4$};
		\end{tikzpicture}
	\end{center}
	\caption{Graphs of the polynomials $E_d(q)$ that give the expected Euler characteristic for the voxel model random cubical complexes, where $d$ is the dimension of the complex.}
	\label{fig:expect}
\end{figure}

To simplify notation, let $E_d(q) = E_{d,0}(q)$.
These polynomials, which give the expected Euler characteristic per unit volume, are particularly nice.
The graphs of the first four of these polynomials are shown in Figure \ref{fig:expect}.

One may ask why these polynomials (other than $E_1$) are not symmetric about $q=1/2$.
To answer this question, let $C_p$ be a $d$-dimensional cubical complex that includes unit $d$-cubes with probability $p$, and similarly for $C_{1-p}$.
Let $C_p^c$ denote the complement of $C_p$ (in the $d$-torus), and observe that $\mu_0(C_p) = -\mu_0(C_p^c)$ by additivity.
While the expected number of unit $d$-cubes in the complex $C_{1-p}$ is the same as that in $C_p^c$, the complex $C_{1-p}$ is a closed set, but $C_p^c$ is open.
Two diagonally-adjacent closed cubes form a connected set, while two diagonally-adjacent open cubes do not.
Thus, we do not expect $\mu_0(C_p)$ to equal $\mu_0(C_{1-p})$, even in absolute value.
For example, if $d=2$ and $p=q=1/2$, the complex $C_p$ is likely to have more holes than connected components, since diagonally-adjacent included squares are connected, but diagonally-adjacent holes are not.

We now examine the polynomial family $E_d$, with the goal of understanding their roots on the interval $[0,1]$.
For $d > 0$, we observe from equation \eqref{eq:expect_dk} that
\begin{equation}\label{eq:ed_simple}
	E_d(q) = \sum_{i=0}^d (-1)^{i+1} \binom{d}{i} q^{2^{d-i}}.
\end{equation}
Thus, the expected Euler characteristic polynomials have a form similar to $(1-q)^d$ and satisfy the following recurrence.

\begin{lemma}\label{lem:recur}
	The expected Euler characteristic polynomials satisfy the recurrence
	\begin{equation}\label{eq:recurrence}
		E_d(q) = E_{d-1}(q^2) - E_{d-1}(q).
	\end{equation}
\end{lemma}
\begin{proof}
	Apply the identity $\binom{d}{i} = \binom{d-1}{i-1} + \binom{d-1}{i}$ to equation \eqref{eq:ed_simple} and simplify.
\end{proof}

\begin{remark}
	The polynomial $E_2(q)$ has a root at $q = \frac{\sqrt{5}-1}{2}$, which is the reciprocal of the golden ratio $\varphi = \frac{1+\sqrt{5}}{2}$.
	In other words, if $q = \frac{1}{\varphi}$ is the probability of \emph{excluding} each pixel from the 2-dimensional voxel model, then the expected Euler characteristic of the random cubical complex is zero.
\end{remark}

\subsection{Interleaving of Roots}

In Figure \ref{fig:expect}, we observe that the roots of $E_{d+1}$ interleave the roots of $E_d$ in the open interval $(0,1)$.
We now prove that this interleaving of roots occurs for all $d$.

\begin{theorem}
	The polynomial $E_d(q)$ has $d+1$ roots in the closed interval $[0,1]$, including roots at the endpoints $q=0$ and $q=1$.
	Moreover, the roots of $E_d$ in the open interval $(0,1)$ interleave the roots of $E_{d+1}$ in that interval.
\end{theorem}
\begin{proof}
	Note that $E_d(q)$ is a polynomial with $d+1$ nonzero coefficients that alternate in sign when arranged by degree.
	Thus, Descartes' Rule of Signs says that $E_d$ has at most $d$ positive roots.
	Furthermore, $E_d(q)$ certainly has a root at $q=0$.
	Denote the roots of $E_d$ as follows: let $q_{d,0} = 0$, and let $0 < q_{d,1} < q_{d,2}, \ldots$ denote the positive roots of $E_d(q)$.
	
	We will prove by induction that $E_d$ has exactly $d+1$ roots in the interval $[0,1]$, including the roots $q_{d,0} = 0$ and $q_{d,d}=1$, and that they satisfy 
	\begin{equation}\label{eq:sqrt_ineq}
		\sqrt{q_{d,i}} < q_{d,i+1}
	\end{equation}
	for all $i = 0, 1, \ldots, d-1$.
	
	\emph{Base case}:
	The roots of $E_1(q)$ are $q_{1,0}=0$ and $q_{1,1}=1$. 
	Thus, $E_1$ has exactly two roots in the interval $[0,1]$, and these roots satisfy inequality \eqref{eq:sqrt_ineq}.
	
	\emph{Induction}:
	Suppose $E_d$ has exactly $d+1$ roots in the interval $[0,1]$, and these roots satisfy inequality \eqref{eq:sqrt_ineq}.
	
	For any $i \in \{0, 1, \ldots, d-1\}$, continuity of $E_d$ implies that $E_d$ does not change sign on the interval $(q_{d,i}, q_{d,i+1})$.
	Without loss of generality, assume that $E_d$ is positive on this interval.
	We exhibit $r,s \in (q_{d,i}, q_{d,i+1})$ such that $E_{d+1}(r) < 0$ and $E_{d+1}(s) > 0$.
	
	If $i>0$, then $q_{d,i} > 0$, so inequality \eqref{eq:sqrt_ineq} implies $\sqrt{q_{d,i}} \in (q_{d,i}, q_{d,i+1})$.
	Then $E_d\left( \sqrt{q_{d,i}} \right) > 0$, and by Lemma \ref{lem:recur} we have
	\begin{equation*}
		E_{d+1}\left( \sqrt{q_{d,i}} \right) = E_d(q_{d,i}) - E_d\left( \sqrt{q_{d,i}} \right) = - E_d\left( \sqrt{q_{d,i}} \right) < 0.
	\end{equation*}
	In this case, let $r=\sqrt{q_{d,i}}$.
	(See Figure \ref{fig:roots}.)
	Similarly, for the $i=0$ case, since $E_d$ is continuous and non-constant there exists $\epsilon \in (0, q_{d,1})$ such that $E_d(\epsilon^2) < E_d(\epsilon)$.
	Thus, $E_{d+1}(\epsilon) = E_d(\epsilon^2) - E_d(\epsilon) < 0$, so we let $r = \epsilon$.
	
	\begin{figure}[b]
		\begin{center}
			\begin{tikzpicture}[xscale=8,yscale=8]
				\draw[->] (-0.5,0) -- (1.1,0) node[right] {$q$};
				\draw (1,0.02) -- (1,-0.02) node[below] {$1$};
				
				\draw[->] (-0.4,-0.2) -- (-0.4,0.2);
				
				\draw (0.081,0.02) -- (.081,-0.02) node[below] {$q_{d,i}$};
				\draw (0.551,0.02) -- (.551,-0.02) node[below] {$q_{d,i+1}$};
				\draw (0.284,0.02) -- (.284,-0.02) node[below] {$\sqrt{q_{d,i}}$};
				\draw (0.743,0.02) -- (.743,-0.02) node[below] {$\sqrt{q_{d,i+1}}$};
				\draw[dashed] (0.417,0.13) -- (0.417,-0.13) node[below] {$q_{d+1,i+1}$};
				
				\draw[blue, thick, domain=0.065:0.58,samples=40] plot (\x, {-sqrt(\x)+4*\x-6*pow(\x,2)+4*pow(\x,4)-pow(\x,8)});
				\node[blue] at (0.09,0.12) {$E_d(q)$};
				\draw[red, thick, domain=0.26:0.76,samples=40] plot (\x, {-\x+4*\x*\x-6*pow(\x,4)+4*pow(\x,8)-pow(\x,16)});
				\node[red] at (0.75,0.12) {$E_d(q^2)$};
				\draw[green, thick, domain=0.2:0.63,samples=40] plot (\x, {sqrt(\x)-5*\x+10*\x*\x-10*pow(\x,4)+5*pow(\x,8)-pow(\x,16)});
				\node[green] at (-0.1,-0.16) {$E_{d+1}(q) = E_d(q^2) - E_d(q)$};
			\end{tikzpicture}
		\end{center}
		\caption{Polynomial $E_{d+1}$ has a root between any two roots of $E_d$.}
		\label{fig:roots}
	\end{figure}
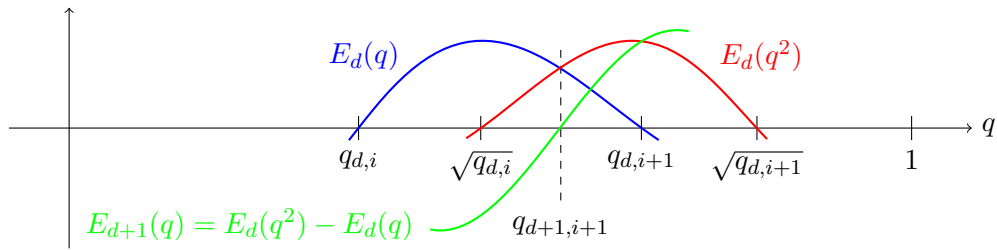
	
	Likewise, if $i<d-1$, then $q_{d,i+1} < 1$, so $q_{d,i+1}^2 \in (q_{d,i}, q_{d,i+1})$ by inequality \eqref{eq:sqrt_ineq}.
	Then $E_d(q_{d,i+1}^2) > 0$, which implies
	\begin{equation*}
		E_{d+1}\left( q_{d,i+1} \right) = E_d\left( q_{d,i+1}^2 \right) - E_d\left( q_{d,i+1} \right) = E_d\left( q_{d,i+1}^2 \right) > 0,
	\end{equation*}
	so let $s = q_{d,i+1}$.
	For the $i=d-1$ case, there exists $\delta \in (q_{d,d-1}, 1)$ such that $E_d(\delta) > E_d(\sqrt{\delta})$.
	This implies $E_{d+1}(\sqrt{\delta}) = E_d(\delta) - E_d(\sqrt{\delta}) > 0$, so $s=\delta$.

	Since $E_{d+1}(r) < 0$ and $E_{d+1}(s) >0$, continuity implies that it must have at least one root in the interval $(r,s) \subset \left( q_{d,i}, q_{d,i+1} \right)$.
	There are $d$ such intervals, so $E_{d+1}$ has at least $d$ roots in the interval $(0,1)$.
	
	Now $E_{d+1}(1)=0$, so $q=1$ is also a root of $E_{d+1}$.
	Thus, $E_{d+1}$ has at least $d+1$ positive roots.
	Since it cannot have any more than $d+1$ positive roots, $E_{d+1}$ has \emph{exactly} $d+1$ positive roots.
	Since $q=0$ is also a root,  $E_{d+1}$ has exactly $d+2$ roots in the interval $[0,1]$, including roots at $0$ and $1$.
	
	For $i>0$ we have $q_{d+1,i} < q_{d,i}$, and so $\sqrt{q_{d+1,i}} < \sqrt{q_{d,i}}$.
	Also, for $i<d$, we have $\sqrt{q_{d,i}} < q_{d+1,i+1}$.
	Thus $\sqrt{q_{d+1,i}} < q_{d+1,i+1}$ for any $i$, which proves that the roots of $E_{d+1}$ in $[0,1]$ satisfy inequality \eqref{eq:sqrt_ineq}.
	
	Lastly, the inequalities in the above paragraph imply that for $i=1, 2, \ldots, d-1$, we have the inequality $q_{d+1,i} < q_{d,i} < q_{d+1,i+1}$.
	Thus, the roots of $E_d$ interleave the roots of $E_{d+1}$ on $(0,1)$.
\end{proof}

\section{Voxel Model: Variance}\label{sec:voxel_var}

We now compute the variance of $\mu_k(C)$ for the voxel model.
Let $V_{d,k}(q)$ denote this variance, normalized by volume.
This normalized variance is a polynomial in $q$, as given in the following theorem.
\begin{theorem}\label{thm:variance}
	For $n>2$, the variance of $\mu_k(C)$, normalized by volume, is a polynomial in $q$ of degree $2^{d-k}$, as follows:
	\begin{equation}\label{eq:variance}
		V_{d,k}(q) = \frac{1}{n^d}\var(\mu_k(C)) 
		= \sum_{i=0}^d \sum_{j=0}^d \sum_{s=0}^d (-1)^{i+j} \binom{i}{k} \binom{j}{k} \binom{d}{s} N_{i,j,s} q^{2^{d-i}+2^{d-j}} \left( q^{-2^{d-s}} - 1 \right),
	\end{equation}
	where
	\begin{equation}
		N_{i,j,s} = \sum_{\ell=0}^s (-1)^{s-\ell} \binom{s}{\ell} \binom{\ell}{i} \binom{\ell}{j} 2^{s+\ell-i-j}.
	\end{equation}
\end{theorem}

\subsection{Proof of Theorem \ref{thm:variance}}

As before, we let $\xi_{X,k}$ be the random variable that indicates the contribution of an open $i$-cube $X$ to $\mu_k(C)$.
That is, $\xi_{X,k} = \mu_{k,i} = (-1)^{i-k}\binom{i}{k}$ if $X$ is included in $C$, and $\xi_{X,k} =0$ otherwise, and we often omit the subscript $k$.
We recall that $\mu_k(C) = \sum_X \xi_X$, where the sum is over all open cubes $X$ of all dimensions in $\lat$.
We express the variance of $\mu_k(C)$ in terms of the indicator variables $\xi_X$ and analyze the contributions of different pairs of cubes.

\begin{proof}
	We can express the variance of $\mu_k(C)$ as follows:
	\begin{align*}
		\var(\mu_k(C)) &= \E\left( (\mu_k(C))^2 \right) - \left( \E(\mu_k(C)) \right)^2 \\
		&= \E \left( \left( \sum_X \xi_X \right)^2 \right) - \left( \E \left( \sum_X \xi_X \right) \right)^2 \\
		&= \sum_{X,Y} \E(\xi_X \xi_Y) - \sum_{X,Y} \E(\xi_X) \E(\xi_Y),
	\end{align*}
	where $\sum_{X,Y}$ means that the sum is over all pairs of open cubes $X$ and $Y$ (of same or different dimensions, including the cases where $X$ and $Y$ are the same open cube) in $\lat$.
	Thus, we have:
	\begin{equation}\label{eq:var1}
		\var(\mu_k(C)) = \sum_{X,Y} \left( \E(\xi_X \xi_Y) - \E(\xi_X) \E(\xi_Y) \right).
	\end{equation}

	We say that $X$ and $Y$ are \b{close} if they are faces of some common $d$-cube in $\lat$; otherwise, we say that $X$ and $Y$ are \b{far}.
	If $X$ and $Y$ are close, then $\xi_X$ and $\xi_Y$ are dependent.
	However, if $X$ and $Y$ are far, then $\xi_X$ and $\xi_Y$ are independent, 
	because the set of $d$-cubes whose selection would cause $X$ to be included in $C$ is disjoint from the set of $d$-cubes whose selection would cause $Y$ to be included.

	Let $i$ be the dimension of open cube $X$ and let $j$ be the dimension of open cube $Y$.
	If $X$ and $Y$ are far, then independence implies that
	\begin{equation*}
		\E(\xi_X \xi_Y) = \E(\xi_X) \E(\xi_Y) = P_i \mu_{k,i} P_j \mu_{k,j}.
	\end{equation*}
	Thus, far pairs of cubes make no contribution to the variance as expressed in equation \eqref{eq:var1}, and it suffices to sum over close cubes.
	That is,
	\begin{equation}\label{eq:var2}
		\var(\mu_k(C)) = \sum_{\text{close } X,Y} \left( \E(\xi_X \xi_Y) - \E(\xi_X) \E(\xi_Y) \right).
	\end{equation}

	In order to systematically sum over all close pairs of cubes, we observe that if $X$ and $Y$ are close, then there is a unique \emph{smallest} cube $S$ of which $X$ and $Y$ are faces (the assumption that $n>2$ is essential here).
	We call $S$ the \b{common cube} of $X$ and $Y$, and we let $s$ be the dimension of $S$.
	Let $N_{i,j,s}$ be the number of pairs $X,Y$ consisting of an $i$-cube and a $j$-cube that are faces of any particular common cube of dimension $s$.
	Let $P_{i,j,s}$ be the probability that both $X$ and $Y$ are included in $C$ when $X$ and $Y$ have a common cube of dimension $s$.
	We can then rewrite the sum in equation \eqref{eq:var2} as:
	\begin{align}
		\var(\mu_k(C)) &= \sum_{i=0}^d \sum_{j=0}^d \sum_{s=0}^d N_s N_{i,j,s} \big( \E(\xi_X \xi_Y) - \E(\xi_X) \E(\xi_Y) \big) \notag \\
		&= \sum_{i=0}^d \sum_{j=0}^d \sum_{s=0}^d N_s N_{i,j,s} \left( P_{i,j,s} \mu_{k,i} \mu_{k,j} - P_i \mu_{k,i} P_j \mu_{k,j} \right) \notag \\
		&= \sum_{i=0}^d \sum_{j=0}^d \sum_{s=0}^d \mu_{k,i} \mu_{k,j} N_s N_{i,j,s} \left( P_{i,j,s} - P_i P_j \right) \label{eq:var3}
	\end{align}
	(In fact, we don't really need the indices of summation to start at $0$; it suffices to start at $i=k$, $j=k$, and $s=\max(i,j)$, but there is no harm in starting each index at $0$ because all unnecessary terms contain binomial coefficients which evaluate to zero.)
	It remains  to compute $N_{i,j,s}$ and $P_{i,j,s}$.

	We use inclusion-exclusion to compute $N_{i,j,s}$. 
	An $s$-cube has $\binom{s}{i} 2^{s-i}$ faces of dimension $i$.
	Thus, the number of pairs $X$ and $Y$, of dimensions $i$ and $j$, that are faces of any $s$-cube $S$, is $\binom{s}{i} 2^{s-i} \binom{s}{j} 2^{s-j}$.
	However, the previous formula counts pairs $X$ and $Y$ that are faces of some common cube smaller than $S$.
	Inclusion-exclusion gives the number of pairs of cubes that are faces of a common cube $S$ and \emph{not} faces of any smaller common cube:
	\begin{align}
		N_{i,j,s} &= \sum_{\ell=0}^s (-1)^{s-\ell} \binom{s}{\ell} 2^{s-\ell} \binom{\ell}{i} 2^{\ell-i} \binom{\ell}{j} 2^{\ell-j} \notag \\
		&= \sum_{\ell=0}^s (-1)^{s-\ell} \binom{s}{\ell} \binom{\ell}{i} \binom{\ell}{j} 2^{s+\ell-i-j}. \label{eq:nijs}
	\end{align}
	
	Now we compute $P_{i,j,s}$.
	Suppose that $S$ is the common cube of $X$ and $Y$.
	Then $X$ and $Y$ are included in $C$ if either of the following two events occur:
	\begin{enumerate}
		\item[(a)] cube $S$ is included, or
		\item[(b)] cube $S$ is not included, but other $d$-cubes containing $X$ and $Y$ are included.
	\end{enumerate}
	Informally, let $\P(S)$ denote the probability that cube $S$ is included in $C$, and let $\P(\neg S)$ be the probability that $S$ is not included in $C$.
	Then:
	\begin{equation}\label{eq:prob}
		P_{i,j,s} = \P(S) + \P(\neg S) \cdot \P(X \mid \neg S) \cdot \P(Y \mid \neg S)
	\end{equation}
	Recall that $P_s = 1-q^{2^{d-s}}$ is the probability of including an $s$-cube.
	The probability of including $X$ without $S$ is the probability of excluding $S$ and including another $d$-cube of which $X$ is a face.
	There are $2^{d-i}-2^{d-s}$ such $d$-cubes; thus
	\begin{equation*}
		\P(X \mid \neg S) = 1- q^{2^{d-i}-2^{d-s}} .
	\end{equation*}
	Therefore equation \eqref{eq:prob} becomes
	\begin{equation}\label{eq:pijs}
		P_{i,j,s} = \left( 1-q^{2^{d-s}} \right) + q^{2^{d-s}} \left( 1- q^{2^{d-i}-2^{d-s}} \right) \left( 1- q^{2^{d-j}-2^{d-s}} \right) .
	\end{equation}
	Routine algebra shows that
	\begin{equation*}
		P_{i,j,s} - P_i P_j = q^{2^{d-i}+2^{d-j}} \left( q^{-2^{d-s}} - 1 \right),
	\end{equation*}
	and the proof is complete.
\end{proof}

\subsection{Variance Polynomials}

The variance polynomials are more complicated than the expected value polynomials previously discussed.
Table \ref{tab:var} lists the variance polynomials $V_{d,k}(q)$ for $d$ and $k$ from $0$ to $3$.
While the expected value polynomials have nonzero coefficients only for terms of the form $q^{2^i}$, the variance polynomials have many more nonzero coefficients.
(Curiously, the coefficient of $q^{13}$ in $V_{3,0}$ is zero.)
Furthermore, these coefficients are not unimodal, nor do they alternate in sign in any clear way.
Figure \ref{fig:var_Euler} displays the graphs of the variance polynomials for the Euler characteristic for low-dimensional complexes, and Figure \ref{fig:var_vol} graphs the variance polynomials for all intrinsic volumes in dimension $d=3$. 

\begin{table}[b]
	{\small \begin{tabular}{c|c c c >{\centering\arraybackslash}m{120pt} }
	& $d=0$ & $d=1$ & $d=2$ & $d=3$ \\
	\hline
	$k=0$ & $q-q^2$ & $q-4q^2+6q^3-3q^4$ & \parbox{100pt}{\centering $q-7q^2+20q^3-13q^4$ \\$-24q^5+28q^6+4q^7-9q^8$} & $q-10q^2+42q^3-30q^4-120q^5+156q^6+60q^7-102q^8+64q^9-120q^{10}-48q^{11}+114q^{12}+12q^{14}+8q^{15}-27q^{16}$ \\
	\hline
	$k=1$ & & $q-q^2$ & $4q-18q^2+28q^3-14q^4$ & $9q-69q^2+192q^3-105q^4-264q^5+276q^6+60q^7-99q^8$ \\
	\hline
	$k=2$ & & & $q-q^2$ & $9q-42q^2+66q^3-33q^4$ \\
	\hline
	$k=3$ & & & & $q-q^2$ \\
	\end{tabular}}
	\caption{Variance polynomials $V_{d,k}(q)$}
	\label{tab:var}
\end{table}

\begin{figure}[b]
	\begin{center}
		\begin{tikzpicture}[xscale=8,yscale=12]
			\draw[->] (-0.2,0) -- (1.2,0) node[right] {$q$};
			\draw (1,0.02) -- (1,-0.02) node[below] {$1$};
			
			\draw[->] (0,-0.1) -- (0,0.3);
			\draw (-0.02,0.1) node[left] {$0.1$} -- (0.02,0.1);
			\draw (-0.02,0.2) -- (0.02,0.2);
			
			\draw[blue, thick, domain=0:1,samples=40] plot (\x, {\x-\x*\x});
			\node[blue] at (0.84,0.2) {$V_{0,0}$};
			\draw[green, thick, domain=0:1,samples=50] plot (\x, {\x-4*\x*\x+6*pow(\x,3)-3*pow(\x,4)});
			\node[green] at (0.5,0.03) {$V_{1,0}$};
			\draw[orange, thick, domain=0:1,samples=60] plot (\x, {\x-7*\x*\x+20*pow(\x,3)-13*pow(\x,4)-24*pow(\x,5)+28*pow(\x,6)+4*pow(\x,7)-9*pow(\x,8)});
			\node[orange] at (0.65,0.14) {$V_{2,0}$};
			\draw[red, thick] (0,0)--(.01,.00904)--(.02,.01633)--(.03,.02211)--(.04,.02660)--(.05,.03003)--(.06,.03260)--(.07,.03450)--(.08,.03592)--(.09,.03703)--(.10,.03796)--(.11,.03886)--(.12,.03985)--(.13,.04103)--(.14,.04249)--(.15,.04431)--(.16,.04653)--(.17,.04920)--(.18,.05235)--(.19,.05598)--(.20,.06011)--(.21,.06471)--(.22,.06976)--(.23,.07523)--(.24,.08105)--(.25,.08719)--(.26,.09357)--(.27,.10013)--(.28,.10679)--(.29,.11347)--(.30,.12008)--(.31,.12655)--(.32,.13280)--(.33,.13873)--(.34,.14428)--(.35,.14936)--(.36,.15392)--(.37,.15789)--(.38,.16121)--(.39,.16385)--(.40,.16576)--(.41,.16692)--(.42,.16732)--(.43,.16696)--(.44,.16585)--(.45,.16401)--(.46,.16149)--(.47,.15831)--(.48,.15456)--(.49,.15029)--(.50,.14558)--(.51,.14053)--(.52,.13523)--(.53,.12979)--(.54,.12430)--(.55,.11888)--(.56,.11364)--(.57,.10868)--(.58,.10410)--(.59,.10001)--(.60,.09648)--(.61,.09360)--(.62,.09141)--(.63,.08998)--(.64,.08931)--(.65,.08943)--(.66,.09031)--(.67,.09192)--(.68,.09419)--(.69,.09704)--(.70,.10036)
				--(.71,.10403)--(.72,.10788)--(.73,.11177)--(.74,.11550)--(.75,.11889)--(.76,.12175)--(.77,.12388)--(.78,.12509)--(.79,.12523)--(.80,.12416)--(.81,.12175)--(.82,.11795)--(.83,.11274)--(.84,.10614)--(.85,.09827)--(.86,.08929)--(.87,.07943)--(.88,.06900)--(.89,.05835)--(.90,.04790)--(.91,.03808)--(.92,.02933)--(.93,.02202)--(.94,.01645)--(.95,.01272)--(.96,.01069)--(.97,.00983)--(.98,.00908)--(.99,.00669)--(1, 0);
			\node[red] at (0.5,0.18) {$V_{3,0}$};
		\end{tikzpicture}
	\end{center}
	\caption{Graphs of the polynomials $V_{d,0}(q)$ that give the variance of the Euler characteristic for random cubical complexes of dimensions $0$ to $3$.}
	\label{fig:var_Euler}
\end{figure}
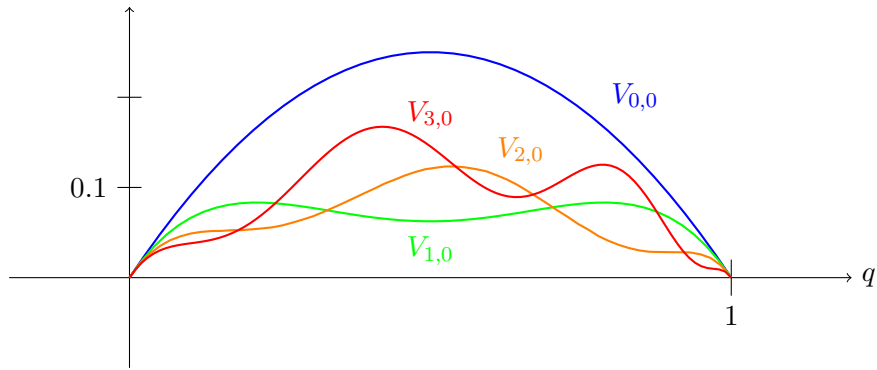

\begin{figure}[htb]
	\begin{center}
		\begin{tikzpicture}[xscale=8,yscale=6]
			\draw[->] (-0.2,0) -- (1.2,0) node[right] {$q$};
			\draw (1,0.02) -- (1,-0.02) node[below] {$1$};
			
			\draw[->] (0,-0.1) -- (0,0.9);
			\foreach \i in {1,...,8}
			{
				\draw (-0.02,0.1*\i) -- (0.02,0.1*\i);
			}
			\node[left] at (-0.02,0.1) {$0.1$};
			\node[left] at (-0.02,0.8) {$0.8$};
			
			\draw[blue, thick, domain=0:1,samples=40] plot (\x, {\x-\x*\x});
			\node[blue] at (0.15,0.2) {$V_{3,3}$};
			\draw[green, thick, domain=0:1,samples=70] plot (\x, {9*\x-42*\x*\x+66*pow(\x,3)-33*pow(\x,4)});
			\node[green] at (0.2,0.5) {$V_{3,2}$};
			\draw[orange,thick] (0,0)--(.01,.08329)--(.02,.15392)--(.03,.21299)--(.04,.26159)--(.05,.30077)--(.06,.33152)--(.07,.35482)--(.08,.37161)--(.09,.38277)--(.10,.38914)--(.11,.39153)--(.12,.39068)--(.13,.38729)--(.14,.38204)--(.15,.37552)--(.16,.36829)--(.17,.36085)--(.18,.35368)--(.19,.34717)--(.20,.34170)--(.21,.33756)--(.22,.33504)--(.23,.33434)--(.24,.33564)--(.25,.33907)--(.26,.34471)--(.27,.35262)--(.28,.36280)--(.29,.37522)--(.30,.38981)--(.31,.40648)--(.32,.42510)--(.33,.44551)--(.34,.46752)--(.35,.49094)--(.36,.51553)--(.37,.54105)--(.38,.56725)--(.39,.59384)--(.40,.62056)--(.41,.64711)--(.42,.67320)--(.43,.69855)--(.44,.72287)--(.45,.74587)--(.46,.76729)--(.47,.78685)--(.48,.80431)--(.49,.81944)--(.50,.83203)--(.51,.84188)--(.52,.84883)--(.53,.85274)--(.54,.85348)--(.55,.85099)--(.56,.84520)--(.57,.83609)--(.58,.82369)--(.59,.80803)--(.60,.78921)--(.61,.76734)--(.62,.74258)--(.63,.71512)--(.64,.68519)--(.65,.65304)--(.66,.61896)--(.67,.58329)--(.68,.54636)--(.69,.50854)--(.70,.47025)
			--(.71,.43188)--(.72,.39387)--(.73,.35665)--(.74,.32066)--(.75,.28633)--(.76,.25409)--(.77,.22434)--(.78,.19747)--(.79,.17382)--(.80,.15369)--(.81,.13733)--(.82,.12493)--(.83,.11657)--(.84,.11229)--(.845,.11164)--(.85,.11197)--(.86,.11541)--(.87,.12227)--(.88,.13203)--(.89,.14403)--(.90,.15742)--(.91,.17111)--(.92,.18382)--(.93,.19400)--(.94,.19982)--(.945,.20044)--(.95,.19915)--(.96,.18956)--(.97,.16824)--(.98,.13201)--(.99,0.07727)--(1,0);
			\node[orange] at (0.7,0.7) {$V_{3,1}$};
			\draw[red, thick]   (0,0)--(.01,.00904)--(.02,.01633)--(.03,.02211)--(.04,.02660)--(.05,.03003)--(.06,.03260)--(.07,.03450)--(.08,.03592)--(.09,.03703)--(.10,.03796)--(.11,.03886)--(.12,.03985)--(.13,.04103)--(.14,.04249)--(.15,.04431)--(.16,.04653)--(.17,.04920)--(.18,.05235)--(.19,.05598)--(.20,.06011)--(.21,.06471)--(.22,.06976)--(.23,.07523)--(.24,.08105)--(.25,.08719)--(.26,.09357)--(.27,.10013)--(.28,.10679)--(.29,.11347)--(.30,.12008)--(.31,.12655)--(.32,.13280)--(.33,.13873)--(.34,.14428)--(.35,.14936)--(.36,.15392)--(.37,.15789)--(.38,.16121)--(.39,.16385)--(.40,.16576)--(.41,.16692)--(.42,.16732)--(.43,.16696)--(.44,.16585)--(.45,.16401)--(.46,.16149)--(.47,.15831)--(.48,.15456)--(.49,.15029)--(.50,.14558)--(.51,.14053)--(.52,.13523)--(.53,.12979)--(.54,.12430)--(.55,.11888)--(.56,.11364)--(.57,.10868)--(.58,.10410)--(.59,.10001)--(.60,.09648)--(.61,.09360)--(.62,.09141)--(.63,.08998)--(.64,.08931)--(.65,.08943)--(.66,.09031)--(.67,.09192)--(.68,.09419)--(.69,.09704)--(.70,.10036)
			--(.71,.10403)--(.72,.10788)--(.73,.11177)--(.74,.11550)--(.75,.11889)--(.76,.12175)--(.77,.12388)--(.78,.12509)--(.79,.12523)--(.80,.12416)--(.81,.12175)--(.82,.11795)--(.83,.11274)--(.84,.10614)--(.85,.09827)--(.86,.08929)--(.87,.07943)--(.88,.06900)--(.89,.05835)--(.90,.04790)--(.91,.03808)--(.92,.02933)--(.93,.02202)--(.94,.01645)--(.95,.01272)--(.96,.01069)--(.97,.00983)--(.98,.00908)--(.99,.00669)--(1, 0);
			\node[red] at (0.43,0.1) {$V_{3,0}$};
		\end{tikzpicture}
	\end{center}
	\caption{Graphs of the polynomials $V_{3,k}(q)$ that give the variance of the intrinsic volumes for random cubical complexes of dimension $3$.}
	\label{fig:var_vol}
\end{figure}
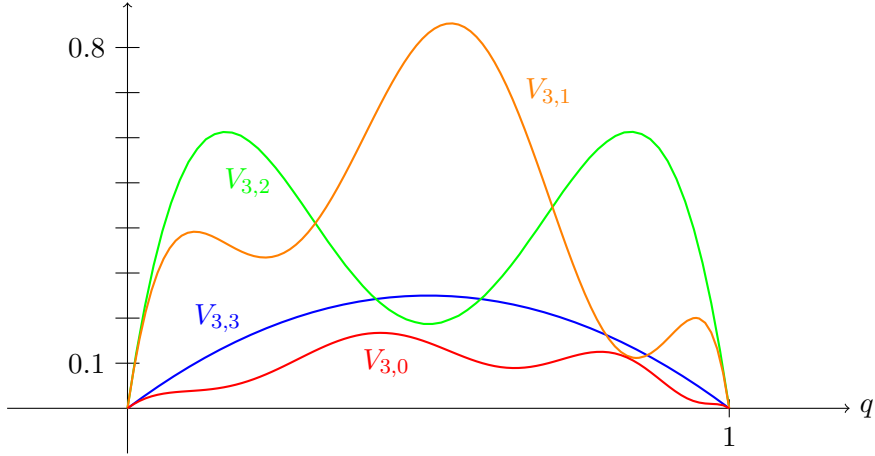

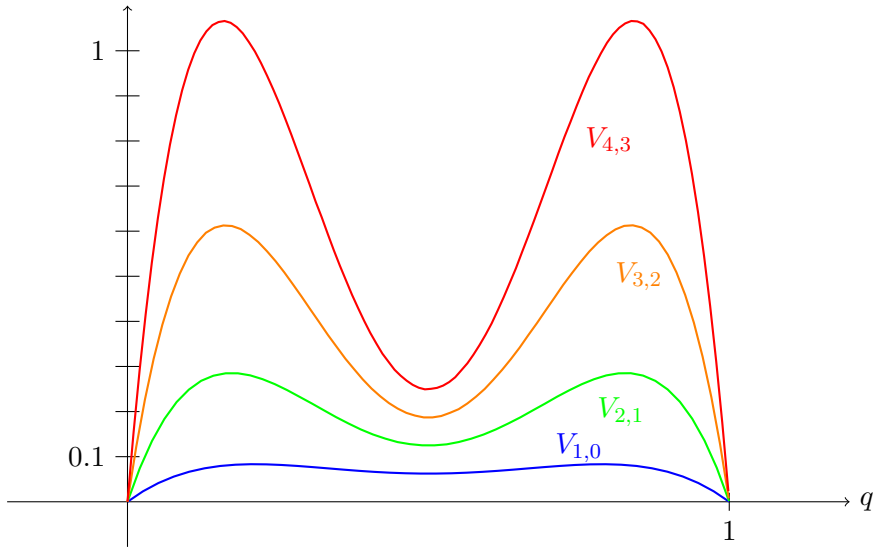
\begin{figure}[htb]
	\begin{center}
		\begin{tikzpicture}[xscale=8,yscale=6]
			\draw[->] (-0.2,0) -- (1.2,0) node[right] {$q$};
			\draw (1,0.02) -- (1,-0.02) node[below] {$1$};
			
			\draw[->] (0,-0.1) -- (0,1.1);
			\foreach \i in {1,...,10}
			{
				\draw (-0.02,0.1*\i) -- (0.02,0.1*\i);
			}
			\node[left] at (-0.02,0.1) {$0.1$};
			\node[left] at (-0.02,1) {$1$};
			
			\draw[blue, thick, domain=0:1,samples=40] plot (\x, {\x-4*\x*\x+6*pow(\x,3)-3*pow(\x,4)});
			\node[blue] at (0.75,0.12) {$V_{1,0}$};
			\draw[green, thick, domain=0:1,samples=50] plot (\x, {4*\x-18*\x*\x+28*pow(\x,3)-14*pow(\x,4)});
			\node[green] at (0.82,0.2) {$V_{2,1}$};
			\draw[orange, thick, domain=0:1,samples=70] plot (\x, {9*\x-42*\x*\x+66*pow(\x,3)-33*pow(\x,4)});
			\node[orange] at (0.85,0.5) {$V_{3,2}$};
			\draw[red, thick, domain=0:1, samples=100] plot (\x, {16*\x-76*\x*\x+120*pow(\x,3)-60*pow(\x,4)});
			\node[red] at (0.8,0.8) {$V_{4,3}$};
		\end{tikzpicture}
	\end{center}
	\caption{Graphs of the polynomials $V_{d,d-1}(q)$ that give the variance of the intrinsic volume $\mu_{d-1}$ for random cubical complexes of dimensions $d=1$ to $4$.}
	\label{fig:var_vol_d1}
\end{figure}

We observe interesting patterns among the variance polynomials in Table \ref{tab:var} when we look along diagonals of the form $V_{d,d-i}$; that is, when we consider intrinsic volumes of codimension $i$.
Most obviously, along the $k=d$ diagonal (codimension $0$), we find $V_{d,d}(q) = q-q^2$.
This is because the intrinsic volume $\mu_d$ is actually the \emph{volume} for $d$-dimensional cubes.
Thus, $\mu_d(C)$ depends only on the number of $d$-cubes selected, and the lower-dimensional faces introduce no dependencies.

Along the codimension-$1$ diagonal, where $k=d-1$, we find polynomials of degree $4$.
Figure \ref{fig:var_vol_d1} illustrates that the graphs of these polynomials all have the same general shape: they are symmetric about $q=1/2$, with local minima at $q=1/2$.
At first glance, these polynomials may appear to be multiples of each other, but this is not so.
In fact, the critical points at which $V_{d,d-1}$ attains local maxima gradually spread outward and approach $\frac{1}{2} \pm \frac{\sqrt{2}}{4}$ as $d \to \infty$.
The following proposition gives explicitly $V_{d,d-1}$ for any $d$.

\begin{proposition}
	When $k=d-1$, the variance polynomials satisfy
	\begin{equation}\label{eq:var_d-1}
		V_{d,d-1}(q) = d^2q - (5d^2-d)q^2 + (8d^2-2d)q^3 - (4d^2-d)q^4.
	\end{equation}
	Each such polynomial has a local minimum at $q=1/2$, and $V_{d,d-1}$ has local maxima at $q = \frac{1}{2} \pm \frac{1}{8d-2}\sqrt{(4d-1)(2d-1)}$.
\end{proposition}
\begin{proof}
	Equation \eqref{eq:var_d-1} can be obtained from  simplifying equation \eqref{eq:variance} when $k=d-1$.
	The statement about local extrema is then elementary calculus.
\end{proof}

Presumably, similar patterns exist along other diagonals of Table \ref{tab:var}.
For example, the variance polynomials for $k=d-2$ also have the same general shape, but are harder to analyze than in the $d-1$ case.
It would also be nice to find a recurrence that connects variance polynomials within a row or column of Table \ref{tab:var}.

\section{Central Limit Theorem}\label{sec:CLT}

Having computed the mean and variance of the intrinsic volumes for the voxel model of random cubical complexes, we now prove a central limit theorem.
Specifically, we show that for fixed $d$ and $q$, the distribution of $\mu_k(C)$ tends toward a normal distribution  as $n$ tends toward infinity.
We rely on an application of Stein's method to sums of random variables with small dependence neighborhoods, detailed in the survey article by Ross \cite{Ross}.

We have seen that $\mu_k(C) = \sum_X \xi_X$ and, in our variance computation, noted that the $\xi_X$ are only locally dependent.
That is, $\xi_X$ and $\xi_Y$ are dependent only if cubes $X$ and $Y$ are faces of some common $d$-cube.
Each $i$-cube is a face of $2^{d-i}$ $d$-cubes, and each $d$-cube has $3^d$ total faces of all dimensions (including itself).
Thus, for any particular $\xi_X$, the number of $\xi_Y$ that are dependent with $\xi_X$ is not greater than $2^{d-i} 3^d \le 6^d$.
Therefore, we say that the size of the dependence neighborhood of $\xi_X$ is bounded by $6^d$.
We emphasize that this bound is independent of $n$, which allows us to prove the following theorem.

\begin{theorem}
	For the $d$-dimensional voxel model and fixed $q \in (0,1)$, as $n \to \infty$ the random variable $\mu_k(C)$ converges in distribution to a normal random variable with mean $E_{d,k}(q)$ and variance $V_{d,k}(q)$.
\end{theorem}
\begin{proof}
	Let $\sigma^2 = \var(\mu_k(C)) = n^d V_{d,k}(q)$.
	Define random variable $W$ as follows:
	\begin{equation*}
		W = \frac{1}{\sigma} \sum_X \left( \xi_X - \E(\xi_X) \right)  = \frac{1}{\sigma}\left( \mu_k(C) - \E(\mu_k(C)) \right).
	\end{equation*}
	Thus $W$ is the standardization (mean zero, unit variance) of the variable $\mu_k(C)$.
	Let $Z$ be a standard normal random variable.

	Since the $\xi_X$ have finite fourth moments and dependency neighborhoods of maximum size $6^d$, we can apply Theorem 3.5 in the survey article by Ross \cite{Ross} to obtain
	\begin{equation}\label{eq:dw}
		d_W(W,Z) \le \frac{6^{2d}}{\sigma^3} \sum_X \E\left| \xi_X - \E(\xi_X) \right|^3 + \frac{\sqrt{26} \cdot 6^{3d/2}}{\sqrt{\pi} \sigma^2} \sqrt{\sum_X \E\left(\xi_X - \E(\xi_X) \right)^4},
	\end{equation}
	where $d_W(W,Z)$ is the Wasserstein distance between the distributions of $W$ and $Z$.
	
	We can also bound the expected value $\E\left| \xi_X - \E(\xi_X) \right|^a$ independently of $n$.
	For $a \in \{ 3, 4 \}$, a simple calculation shows that $\E\left| \xi_X - \E(\xi_X) \right|^a \le 2\binom{d}{k}^a$.
	The lattice $\lat$ contains a total of $(2n)^d$ unit cubes (of all dimensions), so we have:
	\begin{equation*}
		\sum_X \E\left| \xi_X - \E(\xi_X) \right|^a \le (2n)^d \cdot 2\binom{d}{k}^{\! a} \in O(n^d).
	\end{equation*}
	Since $\sigma^2 \in O(n^d)$, both terms in the right side of inequality \eqref{eq:dw} are $O(n^{-d/2})$, which means that the right side of the inequality converges to zero as $n \to \infty$.
	Convergence in Wasserstein metric implies convergence of distribution, so the proof is complete.
\end{proof}

\section{Other Models}\label{sec:models}

\subsection{Closed Faces}

The closed faces model is similar to the voxel model, but it allows closed faces of dimension less than $d$ to be included, regardless of which $d$-cubes are included.
It is constructed by the following process.
Each open $i$-cube (for all $i = 0, 1, \ldots, d$) is selected independently with probability $p$.
Then $C$ is the \emph{closure} of the union of all selected cubes.
(Refer to Figure \ref{fig:models} left.)

To analyze this model, we first determine the probability that any open $i$-cube $X$ is included in $C$.
Observe that $X$ is a face of $3^{d-i}$ cubes of dimensions $i, i+1, \ldots, d$ (here we consider $X$ to be a face of itself).
If \emph{any} of these $3^{d-i}$ cubes are selected, then $X$ will be included in $C$.
Since each open cube is selected independently with probability $p$, the probability that $X$ is included in $C$ is $P_i = 1-q^{3^{d-i}}$, where as before, $q=1-p$.

\begin{proposition}\label{prop:cf}
	For the closed faces model, the expected value of $\mu_k(C)$, normalized by volume, is
	\begin{equation}\label{eq:cf_exp}
		\frac{1}{n^d}\E(\mu_k(C)) = 
		\begin{cases} \sum_{i=k}^{d} (-1)^{i-k+1} \binom{d}{i} \binom{i}{k} q^{3^{d-i}} & \text{if } k < d \\  1-q =p& \text{if } k=d.\end{cases}
	\end{equation}
	The normalized variance is similar to that given in Theorem \ref{thm:variance} except that $P_i = 1-q^{3^{d-i}}$ and
	\begin{equation}\label{eq:cf_pijs}
		P_{i,j,s} = \left( 1-q^{3^{d-s}} \right) + q^{3^{d-s}} \left( 1- q^{3^{d-i}-3^{d-s}} \right) \left( 1- q^{3^{d-j}-3^{d-s}} \right).
	\end{equation}
\end{proposition}
\begin{proof}
	The expected value can be written $\E(\mu_k(C)) = \sum_{i=k}^d N_i P_i \mu_{k,i}$, as in equation \eqref{eq:expect1}.
	From the previous discussion, the probability that any open $i$-cube is included in closed faces model is $P_i = 1-q^{3^{d-i}}$.
	The quantities $N_i$ and $\mu_{k,i}$ are unchanged, so we have
	\begin{equation*}
		\frac{1}{n^d} \E(\mu_k(C)) = \sum_{i=k}^d (-1)^{i-k} \binom{d}{i} \binom{i}{k} \left(1-q^{3^{d-i}} \right).
	\end{equation*}
	We apply identity \eqref{eq:binom_identity} to obtain the formula in equation \eqref{eq:cf_exp}.
	
	For the variance, equation \eqref{eq:variance} holds, with probabilities $P_i$ and $P_{i,j,s}$ specific to the closed faces model.
	We found $P_i$ for this model above.
	To find $P_{i,j,s}$, suppose $i$-cube $X$ is a face of $s$-cube $S$.
	There are $3^{d-i}$ open cubes whose selection causes $X$ to be included in $C$, and $3^{d-s}$ of these also cause $S$ to be included.
	Thus, there are $3^{d-i} - 3^{d-s}$ open cubes whose selection causes $X$, but not $S$, to be included.
	In the notation of equation \eqref{eq:prob}, the probability of including $X$ but not $S$ is 
	\begin{equation*}
		\P(X \mid \neg S) = 1- q^{3^{d-i}-3^{d-s}}.
	\end{equation*}
	Therefore, $P_{i,j,s}$, given by equation \eqref{eq:prob}, becomes
	\begin{equation*}
		P_{i,j,s} = \left( 1-q^{3^{d-s}} \right) + q^{3^{d-s}} \left( 1- q^{3^{d-i}-3^{d-s}} \right) \left( 1- q^{3^{d-j}-3^{d-s}} \right)
	\end{equation*}
	as in equation \eqref{eq:cf_pijs}.
\end{proof}

The expected value polynomials for the closed faces model have a very similar form to those for the voxel model, with $q$ raised to powers of $3$ instead of powers of $2$.
Similar to Proposition \ref{prop:poly}, the expected value of $\mu_k$ for the $d$-dimensional closed faces model is a multiple of the expected value of $\mu_0$ for the $(d-k)$-dimensional closed faces model.
Furthermore, the expected value polynomials for the closed faces model satisfy a recurrence like that in Lemma \ref{lem:recur} and their roots, which depend only on the codimension $d-k$, interleave just as in the voxel model.
Likewise, the variance polynomials for the closed faces model are similar to, but of higher degree than, those for the voxel model.
The distribution of $\mu_k(C)$ for the closed faces model also tends toward a normal distribution as $n \to \infty$, yielding a central limit theorem similar to that in Section \ref{sec:CLT}.

The patterns observed in the voxel and closed faces model suggest that there may be a family of models of random complexes in dimension $d$ such that for any positive integer $m$, there is a model in which the probability of including an open $i$-cube is $1 - q^{m^{d-i}}$.
For the voxel model, $m=2$, and for the closed faces model, $m=3$.
Specific geometric models corresponding to larger $m$ remain to be found.

\subsection{Independent Faces}

Suppose that each open $i$-cube in $\lat$ is included in $C$ independently with probability $P_i$.
We emphasize that in this model, if a high-dimensional cube $X$ is included in $C$, the faces of $X$ are \emph{not} automatically included in $C$.
In other words, $C$ is not necessarily a closed set.
(Refer to Figure \ref{fig:models} center.)

In general, we could choose any set of probabilities $P_0, \ldots, P_d$.
If we choose $P_i = p$ for all $i$, then, because there are no dependencies among the faces, we obtain expected value polynomials that are linear in $p$.
A more interesting model results if $P_i = p^i$.
This implies that lower-dimensional faces have \emph{greater} probability of being included than do higher-dimensional faces, and \emph{all} vertices ($0$-cubes) are included in $C$.
We compute the expected value and variance of $\mu_k(C)$ for this model.

\begin{proposition}
	If each open $i$-cube is included in $C$ independently with probability $p^i$, then the expected value and variance of $\mu_k(C)$, normalized by volume, are:
	\begin{align}
		\frac{1}{n^d}\E(\mu_k(C)) &= \binom{d}{k} p^k (1-p)^{d-k} \label{eq:if_exp} \\
		\frac{1}{n^d}\var(\mu_k(C)) &= \sum_{i=k}^d \binom{d}{i} \binom{i}{k}^2 \left( p^i - p^{2i} \right). \label{eq:if_var}
	\end{align}
\end{proposition}
\begin{proof}
	The expected value can be written $\E(\mu_k(C)) = \sum_{i=k}^d N_i P_i \mu_{k,i}$ (see equation \eqref{eq:expect1}), which becomes
	\begin{equation*}
		\frac{1}{n^d}\E(\mu_k(C)) = \sum_{i=k}^d (-1)^{i-k} \binom{d}{i} \binom{i}{k} p^i.
	\end{equation*}
	We apply the identity $\binom{d}{i} \binom{i}{k} = \binom{d}{k} \binom{d-k}{d-i}$ and re-index the sum to obtain equation \eqref{eq:if_exp}.
	
	Since \emph{all} open cubes are included independently, the variance can be expressed as a sum over all open cubes in $\lat$:
	\begin{equation*}
		\var(\mu_k(C)) = \sum_{\text{open }X} \var(\xi_X).
	\end{equation*}
	The variance of $\xi_X$ for any open $i$-cube $X$ is $\binom{i}{k}^2 (p^i - p^{2i})$.
	Thus,
	\begin{equation*}
		\var(\mu_k(C)) = \sum_{i=k}^d \binom{d}{i} n^d \binom{i}{k}^2 (p^i - p^{2i}),
	\end{equation*}
	and we divide by $n^d$ to obtain equation \eqref{eq:if_var}.
\end{proof}

The intrinsic volumes for the independent faces model satisfy a central limit theorem similar to that of the voxel model in Section \ref{sec:CLT}.
However, for the independent faces model, we observe that the expected value polynomial is always of degree $d$, and the variance polynomial is always of degree $2d$, regardless of $k$.
Thus, the expected value of $\mu_k(C)$ cannot be reduced to a multiple of $\mu_0(C)$ as in the voxel and closed faces models (e.g.\ equation \eqref{eq:d-k}).

We note that unlike the voxel model, the expected values of the independent faces model are never negative.
Furthermore, both the expected value and variance polynomials are unimodal in the independent faces model.

\subsection{Plaquette Model}

In this model, all $(d-1)$-cubes in $\lat$ are included in $C$, and $d$-cubes are included independently with probability $p$.
(Refer to Figure \ref{fig:models} right.)
The normalized expected values and variances of $\E(\mu_k(C))$ are linear and quadratic, respectively, in $p$.

\begin{proposition}
	The expected value and variance of $\mu_k(C)$, normalized by volume, are:
	\begin{align}
		\frac{1}{n^d}\E(\mu_k(C)) &= \begin{cases} (-1)^{d-k} \binom{d}{k}(p-1) & \text{if } k<d \\ p & \text{if }k=d \end{cases} \label{eq:lm_exp} \\
		\frac{1}{n^d}\var(\mu_k(C)) &= \binom{d}{k}^2 (p-p^2). \label{eq:lm_var}
	\end{align}
\end{proposition}
\begin{proof}
	As is now familiar, the expected value can be written $\E(\mu_k(C)) = \sum_{i=k}^d N_i P_i \mu_{k,i}$ (equation \eqref{eq:expect1}).
	However, now $P_i = 1$ if $i < d$ since open cubes of dimension less than $d$ are always included in $C$, and $P_d = p$.
	Thus,
	\begin{equation*}
		\frac{1}{n^d}\E(\mu_k(C)) = \left( \sum_{i=k}^{d-1} (-1)^{i-k} \binom{d}{i} \binom{i}{k} \right) + (-1)^{d-k}\binom{d}{k}p
	\end{equation*}
	We apply identity \eqref{eq:binom_identity} to obtain equation \eqref{eq:lm_exp}.
	
	The variance can be expressed as a sum over all pairs of cubes, as in equation \eqref{eq:var1}:
	\begin{equation*}
		\var(\mu_k(C)) = \sum_{X,Y} \left( \E(\xi_X \xi_y) - \E(\xi_X) \E(\xi_Y) \right).
	\end{equation*}
	In this model, $\E(\xi_X \xi_y) = \E(\xi_X) \E(\xi_Y)$ for all pairs of cubes \emph{except} when $X$ and $Y$ are the \emph{same} $d$-dimensional cube.
	If $X$ is a $d$-cube, then $\E(\xi_X^2) - (\E(\xi_X))^2 = \binom{d}{k}^2 (p-p^2)$.
	Since there are $n^d$ $d$-cubes, the variance reduces to
	\begin{equation*}
		\var(\mu_k(C)) = n^d \binom{d}{k}^2 (p-p^2),
	\end{equation*}
	and we divide by $n^d$ to obtain equation \eqref{eq:lm_var}.
\end{proof}

We notice that the expected values for the plaquette model satisfy a similar relationship to those from the voxel model.
If $E_{d,k}$ is the normalized expected value for $\mu_k(C)$ in the $d$-dimensional plaquette model, then
\begin{equation*}
	E_{d,k} = \binom{d}{k} E_{d-k,0}
\end{equation*}
just as in equation \eqref{eq:d-k}.
However, in the plaquette model, the variance polynomials are all multiples of each other, unlike the variance polynomials in the voxel model.
The plaquette model also yields a central limit theorem similar to that in Section \ref{sec:CLT}.

\section{Further Work}\label{sec:further}

Besides the intrinsic volumes, it is desirable to find the expected homology of random cubical complexes.
However, computing the expected Betti numbers for these complexes seems difficult.
We conjecture that, for most values of $p$ and large $n$, one Betti number dominates all others, and each Betti number gives way to the next as $p$ approaches the zeros of the expected Euler characteristic function.
Such is the situation observed in other studies of the topology of random structures \cite{AdTa, Kahle11, Kahle13}.

\section{Acknowledgements}\label{sec:ack}

The authors gratefully acknowledge the support of the Institute for Mathematics and its Applications (IMA). 
This work was initiated at the IMA during the workshop on Topological Data Analysis in September 2013,
and the second author was a postdoctoral fellow during the IMA's annual program on Scientific and Engineering Applications of Algebraic Topology.


\bibliographystyle{amsalpha}

\begin{thebibliography}{99}

\bibitem{AdTa} Robert J. Adler and Jonathan E. Taylor, {\em Random Fields and Geometry}, Springer, 2007.

\bibitem{Aizenman} Michael Aizenman, Chayes, J. T., Chayes, L., Fröhlich, J., and Russo, L., ``On a Sharp Transition from Area Law to a Perimeter Law in a System of Random Surfaces,'' {\em Communications in Mathematical Physics}, vol. 92, March 1983, p. 19--69. 

\bibitem{BGW} Yuliy Baryshnikov, Robert Ghrist, and Matthew Wright, ``Hadwiger's Theorem for Definable Functions,'' {\em Advances in Mathematics}, vol. 245,  October 2013, p. 573--586.

\bibitem{CCFK} Daniel Cohen, Armindo Costa, Michael Farber, and Thomas Kappeler, ``Topology of Random 2-Complexes,'' {\em Discrete and Computational Geometry}, vol. 47, issue 1, January 2012, p. 117--149.

\bibitem{CGR} Justin Curry, Robert Ghrist, and Michael Robinson, ``Euler Calculus and its Applications to Signals and Sensing,'' {\em Proceedings of Symposia in Applied Mathematics}, AMS, 2012.

\bibitem{gray} Stephen B. Gray, ``Local properties of binary images in two dimensions,'' {\em IEEE Transactions on Computers}, vol. 20, issue 5, May 1971, p. 551--561.

\bibitem{GrHo} Geoffrey R. Grimmett and Alexander E. Holroyd, ``Plaquettes, Spheres, and Entanglement,'' {\em Electronic Journal of Probability}, vol. 15 (2010), p. 1415--1428.

\bibitem{GKMSS} Ralph Guderlei, Simone Klenk, Johannes Mayer, Volker Schmidt, and Evgeny Spodarev, ``Algorithms for the computation of the Minkowski functionals of deterministic and random polyconvex sets,'' {\em Image and Vision Computing}, vol 25, no. 4 (2007), p. 464--474.
 
\bibitem{Kahle11} Matthew Kahle, ``Random geometric complexes'', {\em Discrete and Computational Geometry}, vol. 45, issue 3, April 2011, p. 553--573.

\bibitem{Kahle13} Matthew Kahle, ``Topology of random simplicial complexes: a survey,'' {\em AMS Contemporary Volumes in Mathematics}, 2014.

\bibitem{KaMe} Matthew Kahle and Elizabeth Meckes, ``Limit theorems for Betti numbers of random simplicial complexes,'' {\em Homology, Homotopy, and Applications}, vol. 15, no. 1 (2013), p. 343--374.

\bibitem{KlRo} Daniel A. Klain and Gian-Carlo Rota, {\em Introduction to Geometric Probability}, Cambridge University Press, 1997.

\bibitem{LMB} Xiaoxing Li, Paulo R.S. Mendon\c{c}a, Rahul Bhotika, ``Texture Analysis Using Minkowski Functionals,'' {\em Medical Imaging 2012: Image Processing}, Proceedings of SPIE, vol. 8314.

\bibitem{LiMe} Nathan Linial and Roy Meshulam, ``Homological Connectivity of Random 2-Complexes,'' {\em Combinatorica}, vol. 26 (2006), p. 475--487.

\bibitem{MeWa} Roy Meshulam and N. Wallach, ``Homological Connectivity of Random $k$-Dimensional Complexes,'' {\em Random Structures \& Algorithms}, vol. 34, no. 3 (2000), p. 408--417.

\bibitem{RW} Eitan Richardson  and Michael Werman. ``Efficient classification using the Euler characteristic,'' {\em Pattern Recognition Letters}, vol. 49 (2014), p. 99--106.

\bibitem{Ross} Nathan Ross, ``Fundamentals of Stein's Method,'' {\em Probability Surveys}, vol. 8, 2011, p. 210--293. 

\bibitem{SON} Katja Schladitz, Joachim Ohser, and Werner Nagel, ``Measurement of intrinsic volumes of sets observed on lattices,'' {\em Discrete Geometry for Computer Imagery}, 2006, p. 247--258.

\bibitem{ScWe} Rolf Schneider and Wolfgang Weil, {\em Stochastic and Integral Geometry}, Springer, 2009.

\bibitem{Schanuel} Stephen H. Schanuel, ``What is the Length of a Potato?'' {\em Lecture Notes in Mathematics}, Springer, 1986, p. 118--126.

\bibitem{Svane} Anne Marie Svane, ``Estimation of Intrinsic Volumes from Digital Grey-scale Images,'' {\em Journal of Mathematical Imaging and Vision}, vol. 49, no. 2 (2014), p. 352--376.

\bibitem{vdD} Lou van den Dries, {\em Tame Topology and O-Minimal Structures}, Cambridge University Press, 1998.

\bibitem{Wright} Matthew L. Wright, ``Hadwiger Integration of Random Fields,'' {\em Topological Methods in Nonlinear Analysis}, to appear (2014).
\end{thebibliography}

\end{document}